\documentclass[11pt,amssymb,amstex]{amsart}
\usepackage{graphicx}
\usepackage{epstopdf}

\newtheorem{theorem}{Theorem}[section]

\newtheorem{proposition}[theorem]{Proposition}
\newtheorem{definition}[theorem]{Definition}
\newtheorem{remark}[theorem]{Remark}
\newtheorem{conjecture}[theorem]{Conjecture}
\numberwithin{equation}{section}

\begin{document}

\title{Mutation and Chaos in Nonlinear Models of Heredity}

\author{Nasir Ganikhodjaev}
\address{Nasir Ganikhodjaev\\
Department of Computational \& Theoretical Sciences \\
Faculty of Sciences, International Islamic University Malaysia\\
P.O. Box, 141, 25710, Kuantan\\
Pahang, Malaysia} \email{{\tt nasirgani@@hotmail.com}}

\author{Mansoor Saburov}
\address{Mansoor Saburov\\
Department of Computational \& Theoretical Sciences \\
Faculty of Science, International Islamic University Malaysia\\
P.O. Box, 141, 25710, Kuantan\\
Pahang, Malaysia} \email{{\tt msaburov@@gmail.com}}

\author{Ashraf Mohamed Nawi}
\address{Ashraf Mohamed Nawi\\
Department of Computational \& Theoretical Sciences \\
Faculty of Science, International Islamic University Malaysia\\
P.O. Box, 141, 25710, Kuantan\\
Pahang, Malaysia}
\email{{\tt ashraf{\_}nawi@yahoo.com}}

\begin{abstract}
{In this short communication, we shall explore a nonlinear discrete dynamical system that naturally occurs in population systems to describe a transmission of a trait from parents to their offspring. We consider  a Mendelian inheritance for a single gene with three alleles and assume that to form a new generation, each gene has a possibility to mutate, that is, to change into a gene of the other kind. We investigate the derived models. A numerical simulation assists us to get some clear picture about chaotic behaviors of such models.

}

\end{abstract}

\maketitle

\section{Introduction}
Recently chaotic dynamical systems are very popular in science and engineering. Besides the original definition of Li-Yorke chaos in \cite{LY}, there have been various definitions for "chaos" in the literature, and the most often used one is given by Devaney in \cite{Dev}. Although there is no universal definition for chaos, the essential feature of chaos is sensitive dependence on initial conditions so that the eventual behavior of the dynamics is unpredictable. The theory and methods of chaotic dynamical systems have been of fundamental importance not only in mathematical sciences, but also in physical, engineering, {\textbf{biological}}, and even economic sciences. In this paper, a chaos would be understood in the sense of Li-Yorke (see \cite{FBlanEGlasSKloyAMaa, FBlanVLop}).

In this short communication, we introduce and examine a family of nonlinear discrete dynamical systems that naturally occurs to describe a transmission of a trait from parents to their offspring. Here, we shall present some essential analytic and numerical results on dynamics of such models by avoiding some technical parts of the proofs. In upcoming papers, we shall explore the deepest investigation on asymptotic behaviors of the family of nonlinear dynamical systems by providing the proofs in details.  

As the first example, we consider a Mendelian inheritance of a single gene with two alleles $\textbf{{A}}$ and $\textbf{{a}}$ (see \cite{MLR}). Let an element $\textbf{{x}}$ represent a gene pool for a population and it have been expressed a linear combination of the alleles $\textbf{A}$ and $\textbf{{a}}$
\begin{eqnarray}\label{twoalleles}
\textbf{{x}}= x_1\textbf{{A}}+x_2\textbf{{a}}
\end{eqnarray} 
where, $0\leq x_1,x_2 \leq1 $ and $x_1+x_2=1.$ Then, $x_1, x_2$  are the percentage of the population  which carries the alleles $\textbf{A}$ and $\textbf{a}$ respectively. The rules of  the Mendelian inheritance indicate that the next generation
$$\textbf{{x}}^{\prime}= x^{\prime}_1\textbf{{A}}+x^{\prime}_2\textbf{{a}}$$
represents a gene pool for the population which carries the alleles $\textbf{A}$ and $\textbf{a}$ respectively, where
\begin{eqnarray}\label{TwoAllelGeneral}
\begin{cases}
x_1^{\prime}=P_{AA,A} x_1^2+2P_{Aa,A}x_1x_2+ P_{aa,A}x_2^2\\
x_2^{\prime}=P_{AA,a} x_1^2+2P_{Aa,a}x_1x_2+ P_{aa,a}x_2^2
\end{cases}
\end{eqnarray}

Here, $P_{AA,A}$ (resp. $P_{AA,a}$) is the probability that the child receives the allele $\textbf{A}$ (resp. $\textbf{a}$) from parents with the allele $\textbf{A}$;  $P_{Aa,A}$ (resp. $P_{Aa,a}$)  is the probability that the child receives the allele $\textbf{A}$ (resp. $\textbf{a}$) from parents with the alleles $\textbf{A}$ and $\textbf{a}$ respectively; and $P_{aa,A}$ (resp. $P_{aa,a}$)  is the probability that the child receives the allele $\textbf{A}$ (resp. $\textbf{a}$) from parents with allele $\textbf{a}$. It is evident that 
$$P_{..,A}+P_{..,a}=1, \quad P_{Aa,A}=P_{aA,A}, \quad  P_{Aa,a}=P_{aA,a}, \quad  x_1^{\prime}+x_2^{\prime}=1.$$

Thus, \eqref{TwoAllelGeneral} is a nonlinear dynamical system acting on the one dimensional symplex
$$S^1=\left\{(x_1,x_2)\in \mathbb{R}^2: x_1,x_2\geq 0 ,  x_1+x_2=1\right\}$$
that  describes  the distribution of the next generation which carries the alleles $\textbf{A}$ and $\textbf{a}$ respectively,  if the distribution of the current generation are known. 

Recall that in the Mendelian inheritance case, i.e., $P_{AA,A}=P_{aa,a}=1$ and $P_{AA,a}=P_{aa,A}=0$, the dynamical system \eqref{TwoAllelGeneral} has the following form
\begin{eqnarray}\label{TwoAllel01}
\begin{cases}
x_1^{\prime}= x_1^2+2P_{Aa,A}x_1x_2\\
x_2^{\prime}=2P_{Aa,a}x_1x_2+ x_2^2.
\end{cases}
\end{eqnarray}     
     
We assume that prior to a formation of a new generation each gene has a possibility to mutate, that is, to change into a gene of the other kind. Specifically, we assume that for each gene the mutation $\textbf{A}\rightarrow \textbf{a}$ occurs with probability $\alpha,$ and $\textbf{a}\rightarrow \textbf{A}$ occurs with probability $\beta.$ We assume that "\textit{the mutation occurs if only if both parents have the same allele}". Then, we have that $ P_{AA,a}= \alpha, \ P_{aa,A}=\beta$, $P_{AA,A}=1-\alpha, \ P_{aa,a}=1-\beta$ and the dynamical system \eqref{TwoAllelGeneral} has the following form
\begin{eqnarray}\label{TwoAllilalphabeta}
V:\begin{cases}
x_1^{\prime}=(1-\alpha) x_1^2+2P_{Aa,A}x_1x_2+ \beta x_2^2\\
x_2^{\prime}=\alpha x_1^2+2P_{Aa,a}x_1x_2+ (1-\beta)x_2^2.
\end{cases}
\end{eqnarray}

An operator $V:S^1\rightarrow S^1$ given by \eqref{TwoAllilalphabeta} is called a quadratic stochastic operator.

We introduce some standard terms in the theory of a discrete dynamical system $V:X\rightarrow X$. 

A sequence $\{\textbf{x}^{(n)}\}_{n=0}^\infty$ is called a trajectory of $V$ starting from an initial point $\textbf{x}^{0},$ where $\textbf{x}^{(n)}=V(\textbf{x}^{(n-1)})$ for any $n\in\mathbb{N}$. Recall that a  point $\textbf{x}$ is called a fixed point of  $V$ if $V(\textbf{x})=\textbf{x}.$ We denote a set of all fixed points by $Fix(V)$.
A dynamical system $V$ is called regular if a trajectory $\{\textbf{x}^{(n)}\}_{n=0}^\infty$ converges for any initial point $\textbf{x}$.
Note that if $V$ is regular, then the limiting point is a fixed point of $V.$ Thus, in a regular system, the fixed point of dynamical system $V$  describes a long run behavior of the trajectory of $V$ starting from any initial point.
The biological treatment of the regularity of dynamical system $V$  is rather clear: in a long run time, the distribution of species in the next generation coincide with distribution of species in the current generation, i.e., stable.

A fixed point set and an omega limiting set of quadratic stochastic operators (QSO) were deeply studied in \cite{K}-\cite{J}, and they play an important role in many applied problems \cite{Br,Lyu}. In the paper \cite{GRMFRU}, it was given a long self-contained exposition of recent achievements and open problems in the theory of quadratic stochastic operators.  

\begin{definition}
 A dynamical system $V:X\to X$ is said to be ergodic if the limit
\begin{equation}\label{erg}
\lim_{n\rightarrow \infty } \frac{1}{n} \sum_{k=0}^{n-1} V^k(\textbf{x})
\end{equation}
exists for any $\textit{\textbf{x}}\in X$.
\end{definition}

Based on some numerical calculations, S.Ulam  has conjectured \cite{U} that any QSO $V:S^{m-1}\to S^{m-1}$ acting on the finite dimensional simplex is ergodic.
However, in \cite{Z},  M.Zakharevich showed that, in general, Ulam's conjecture is false. Namely, M.Zakharevich showed that the following QSO $V_0:S^2\to S^2$ is not ergodic
\begin{eqnarray}\label{ExampleofZakharevich}
V_0:
\begin{cases}
x'_1=x_1^2+2x_1x_2\\
x'_2=x_2^2+2x_2x_3\\
x'_3=x_3^2+2x_3x_1
\end{cases}. 
\end{eqnarray}

In \cite{GZ}, Zakharevich's result was generalized in the class of Volterra QSO. Moreover, in \cite{GZ}, it was given a necessary and sufficient
condition of being non-ergodicity of Volterra QSO defined on $S^2.$ 

We define the $k$-th order Cesaro mean by the following formula $$Ces_k^{(n)}\left(\textbf{x},V\right)=\frac{1}{n}\sum\limits_{i=0}^{n-1}Ces_{k-1}^{(i)}\left(\textbf{x},V\right)$$
where $k\geq 1$ and $Ces_{0}^{(n)}\left(\textbf{x},V\right)=V^n(\textbf{x})$. It is clear that the first order Cesaro mean $Ces_1^{(n)}\left(\textbf{x},V\right)$ is nothing but $\frac{1}{n}\sum\limits_{k=0}^{n-1}V^k(\textbf{x})$. In this manner, Zakharevich's result says that the first order Cesaro mean $\left\{Ces_1^{(n)}\left(\textbf{x},V_0\right)\right\}_{n=0}^\infty$ of the trajectory  of operator \eqref{ExampleofZakharevich} diverges for any initial point except fixed points. Surprisingly, in \cite{MSabu2007}, it was shown that any order Cesaro mean $\left\{Ces_k^{(n)}\left(\textbf{x},V_0\right)\right\}_{n=0}^\infty$, where $k\geq 1$, of the trajectory  of operator \eqref{ExampleofZakharevich} diverges for any initial point except fixed points.
This leads to a conclusion that the operator $V_0$ given by \eqref{ExampleofZakharevich} might have unpredictable behavior.
 
In fact, in \cite{MSabu2012LY}, it was proven that an operator $V_0$ given by \eqref{ExampleofZakharevich} exhibits the Li-Yorke chaos. It is worth pointing out that some other properties of Volterra QSO were studied in \cite{MSabu2012, MSabu2013}.

Note that if QSO is regular, then it is ergodic. However, the reverse implication is not always true.  It is known that the dynamical system \eqref{TwoAllilalphabeta} is either regular or converges to a periodic-2 point (see \cite{Lyu2}). Therefore, in 1D simplex, any QSO is ergodic. In other words, a mutation in population system having \textbf{a single gene with two alleles} always exhibits an ergodic behavior (or almost regular).  It is of independent interest to study a mutation in population system having \textbf{a single gene with three alleles}. In the next section, we consider an inheritance of {a single gene with three alleles} $\textbf{a}_1,\textbf{a}_2$ and $\textbf{a}_3$ and show that a nonlinear dynamical system corresponding to the mutation exhibits a non-ergodic behavior (or Li-Yorke chaos).

\section{Inheritance for a Single Gene with Three Alleles}
 In this section, we shall derive a mathematical model of a inheritance of a single gene with three alleles. In this case, an element $\textbf{x}$ represents a population if its expression $\textbf{{x}}= x_1a_1+x_2a_2+x_3a_3$, as a linear combination of the alleles $\textbf{a}_1,\textbf{a}_2$ and $\textbf{a}_3$, satisfies the following conditions $ 0\leq x_1,x_2,x_3\leq1$ and $x_1+x_2+x_3=1.$ Then $x_1, x_2, x_3$  are the percentage of the population  which carries the alleles $\textbf{a}_1,\textbf{a}_2$ and $\textbf{a}_3$ respectively. 
 
We assume that prior to a formation of a new generation each gene has a possibility to mutate, that is, to change into a gene of the other kind. We assume that the mutation occurs if \textit{both parents have the same alleles}. Specifically, we will consider two types of the simplest mutations:

\begin{enumerate}
\item {Assume that mutations $a_1\rightarrow a_2$, $a_2 \rightarrow a_3$, and $a_3\rightarrow a_1$ occur with probability $\alpha$};

\item {Assume that mutations $a_1\rightarrow a_3$, $ a_2\rightarrow a_1$, 
and $a_3\rightarrow a_2$  occur with probability $\alpha.$} 
\end{enumerate}

Then the corresponding dynamical systems are defined on the two-dimensional simplex
$$S^2 = \{(x_1,x_2,x_3)\in \mathbb{R}^3:  x \geq 0, x_2\geq  0, x_3\geq  0,  x_1+x_2+x_3=1\}.$$
In the first mutation, we have
\begin{eqnarray}\label{Valpha}
V_\alpha:
\begin{cases}
x'_1=(1-\alpha)x_1^2+2x_1x_2+\alpha x_2^2\\
x'_2=(1-\alpha)x_2^2+2x_2x_3+\alpha x_3^2\\
x'_3=(1-\alpha)x_3^2+2x_3x_1+\alpha x_1^2
\end{cases}.
\end{eqnarray}
In the second mutation, we have
\begin{eqnarray}\label{Walpha}
W_\alpha:
\begin{cases}
x'_1=(1-\alpha)x_1^2+2x_1x_2+\alpha x_3^2\\
x'_2=(1-\alpha)x_2^2+2x_2x_3+\alpha x_1^2\\
x'_3=(1-\alpha)x_3^2+2x_3x_1+\alpha x_2^2
\end{cases}.
\end{eqnarray}

In both cases, if $\alpha=0$, i.e., if a mutation does not occur, then a dynamical systems \eqref{Valpha} and \eqref{Walpha} coincide with Zakharevich's operator \eqref{ExampleofZakharevich}. As we already mentioned, Zakharevich's operator exhibits a Li-Yorke chaos (see \cite{MSabu2012LY}). 

Let $\alpha=1$.  In the first case, an operator $V_1$ is a permutation of Zakharevich's operator \eqref{ExampleofZakharevich}. An omega limiting set of the permutation of Zakharevich's operator was studied in \cite{JURU}. By means of results \cite{Z}, \cite{MSabu2007}, \cite{MSabu2012LY}, and \cite{JURU}, one can show that the operator $V_1$ is non-ergodic as well as a Li-Yorke chaos. Let $\alpha=1$. In the second case, an operator $W_1$ is a permutation of an operator which was studied in \cite{GR1989}. By applying the same method which was given in \cite{GR1989}, we may show that the operator $W_1$ is regular. 

It is easy to check that $V_\alpha=(1-\alpha)V_0+\alpha V_1$ and $W_\alpha=(1-\alpha)W_0+\alpha W_1$.  In other words, in the first case, a mutation operator is a convex combination of two Li-Yorke chaotic operators and in the second case, a mutation operator is a convex combination of a Li-Yorke chaotic and  regular operators. Both operators $V_\alpha$ and $W_\alpha$ were not studied in either paper \cite{JURU} or \cite{GR1989}.  In the next section, we are going to present some essential analytic and numerical results on dynamics of $V_\alpha$ and $W_\alpha$, by avoiding some technical parts of the proofs. In upcoming papers, we shall explore the deepest investigation on asymptotic behaviors of $V_\alpha$ and $W_\alpha$ by providing the proofs in details.

\section{ Attractors: Analytic and Numerical Results}

In this section, we shall study dynamics of the operators \eqref{Valpha} and \eqref{Walpha}. Here, we shall present some pictures of attractors of the operators \eqref{Valpha} and \eqref{Walpha}.

\subsection{Analytic Results on Dynamics of $V_\alpha$}

We are aiming to present some analytic results on dynamics of $V_\alpha:S^2\to S^2$:
\begin{eqnarray}\label{V_alpha}
V_\alpha:
\begin{cases}
x'_1=(1-\alpha)x_1^2+2x_1x_2+\alpha x_2^2\\
x'_2=(1-\alpha)x_2^2+2x_2x_3+\alpha x_3^2\\
x'_3=(1-\alpha)x_3^2+2x_3x_1+\alpha x_1^2
\end{cases},
\end{eqnarray}
where $V_\alpha(x)=x'=(x'_1,x'_2,x'_3)$ and $0<\alpha<1$. As we already mentioned, this operator can be written in the following form: $V_\alpha=(1-\alpha)V_0+\alpha V_1$ for any $0<\alpha<1$, where
$$
V_0:
\begin{cases}
x'_1=x_1^2+2x_1x_2\\
x'_2=x_2^2+2x_2x_3\\
x'_3=x_3^2+2x_3x_1
\end{cases}, \quad \quad \quad 
V_1:
\begin{cases}
x'_1=x_2^2+2x_1x_2\\
x'_2=x_3^2+2x_2x_3\\
x'_3=x_1^2+2x_3x_1
\end{cases}.
$$

Let 
$$
P=\left(
\begin{array}{ccc}
0 & 1 & 0\\
0 & 0 & 1\\
1 & 0 & 0
\end{array}
\right)
$$ be a permutation. The proofs of the following results are straightforward.

\begin{proposition}\label{permutationandV}
Let $V_\alpha :S^2\to S^2$ be a quadratic stochastic operator given by \eqref{V_alpha}, where $\alpha\in(0,1)$. Let $Fix(V_\alpha)$ and $\omega(x^0)$ be sets of fixed points and omega limiting points of $V_\alpha$, respectively. Then the following statements hold true.
\begin{itemize}
\item[(i)] Operators $P$ and $V_\alpha$ are commutative, i.e., $P\circ V_\alpha=V_\alpha\circ P$;
\item[(ii)] If $x\in Fix(V_\alpha)$ then $Px\in Fix(V_\alpha)$;
\item[(iii)] If $Fix(V_\alpha)$ is a finite set then $|Fix(V_\alpha)|\equiv 1 \ (mod \ 3)$;
\item[(iv)] One has that $P(\omega(x^0))=\omega(Px^0)$, for any $x^0\in S^2$.  
\end{itemize}
\end{proposition}  

We are aiming to study the fixed point set $Fix(V_\alpha)$, where  $\alpha\in(0,1)$. It is worth mentioning that $Fix(V_0)=\{e_1,e_2,e_3, C\}$ and $Fix(V_1)=\{C\}$, where $e_1,e_2,e_3$ are vertexes of the simplex $S^2$ and $C=\left(\frac{1}{3},\frac{1}{3},\frac{1}{3}\right)$ is a center of the simplex $S^2$.

Recall \cite{Lyu} that a fixed point $x^0\in Fix(V_\alpha)$ is non-degenerate if and only if the following determinant is nonzero at the fixed point $x^0$:
\begin{eqnarray}
\left|
\begin{array}{ccc}
\frac{\partial x'_1}{\partial x_1}-1 & \frac{\partial x'_1}{\partial x_2} & \frac{\partial x'_1}{\partial x_3}\\
\frac{\partial x'_2}{\partial x_1} & \frac{\partial x'_2}{\partial x_2}-1 & \frac{\partial x'_2}{\partial x_3}\\
1 & 1 & 1
\end{array}
\right|\neq 0
\end{eqnarray}

\begin{proposition}\label{uniquefixedpoint}
Let $V_\alpha :S^2\to S^2$ be a quadratic stochastic operator given by \eqref{V_alpha}, where $\alpha\in(0,1)$. Let $C=\left(\frac{1}{3},\frac{1}{3},\frac{1}{3}\right)$ be a center of the simplex $S^2$. Then the following statements hold true:
\begin{itemize}
\item[(i)] All fixed points are non-degenerate;
\item[(ii)] One has that $Fix(V_\alpha)=\{C\}$.
\end{itemize}
\end{proposition}

\begin{proof}
(i). Let $x\in Fix(V_\alpha)$ be a fixed point. One can easily check that 
\begin{eqnarray*}
\left|
\begin{array}{ccc}
\frac{\partial x'_1}{\partial x_1}-1 & \frac{\partial x'_1}{\partial x_2} & \frac{\partial x'_1}{\partial x_3}\\
\frac{\partial x'_2}{\partial x_1} & \frac{\partial x'_2}{\partial x_2}-1 & \frac{\partial x'_2}{\partial x_3}\\
1 & 1 & 1
\end{array}
\right|=4(1-\alpha+\alpha^2)(x_1x_2+x_1x_3+x_2x_3)+2\alpha-1
\end{eqnarray*}

If $\frac{1}{2}\leq \alpha <1$ then the above expression is positive. This means that, in the case $\frac{1}{2}\leq \alpha <1$, all fixed point are non-degenerate. 

Let us consider the case $0<\alpha<\frac{1}{2}.$ In this case, the above expression is equal to zero if and only if $x_1x_2+x_1x_3+x_2x_3=\frac{1-2\alpha}{4(1-\alpha+\alpha^2)}$. Since $x_1+x_2+x_3=1$, we have that $x_1^2+x_2^2+x_3^2=\frac{1+2\alpha^2}{2(1-\alpha+\alpha^2)}.$ 

Without loss of generality, we may assume that $x_1\geq \max\{x_2,x_3\}$ (See Proposition \ref{permutationandV} (i)). 

Let $x_2\geq x_3$. Since $x\in Fix(V_\alpha)$, we have that $x_2=(1-\alpha)x_2^2+2x_2x_3+\alpha x_3^2.$ We then obtain that
\begin{eqnarray*}
\frac{1+2\alpha^2}{2(1-\alpha+\alpha^2)}&=&x_1^2+x_2^2+x_3^2\\
&=&x_1^2+[(1-\alpha)(x_2^2+2x_2x_3)+\alpha(x_3^2+2x_2x_3)]^2+x_3^2\\
&\leq&x_1^2+[(1-\alpha)x_2+\alpha x_3]^2+x_3^2\\
&<&x_1^2+x_2^2+x_3^2=\frac{1+2\alpha^2}{2(1-\alpha+\alpha^2)}
\end{eqnarray*} 
This is a contradiction. In the similar manner, one can have a contradiction whenever $x_3\geq x_2$. This shows that, in the case $0<\alpha<\frac{1}{2}$, all fixed points are non-degenerate.

(ii). We want to show that $Fix(V_\alpha)=\{C\}$. It is clear that $V_\alpha(\partial S^2)\subset intS^2$. This means that the operator $V_\alpha$ does not have any fixed point on the boundary $\partial S^2$ of the simplex $S^2$, i.e., $Fix(V_\alpha)\cap \partial S^2=\emptyset$. Moreover, all fixed point are non-degenerate. Due to Theorem 8.1.4 in \cite{Lyu}, $|Fix(V_\alpha)|$ should be odd. On the other hand, due to Corollary 8.1.7 in \cite{Lyu}, one has that $|Fix(V_\alpha)|\leq 4$. Proposition \ref{permutationandV}, (iii) yields that $|Fix(V_\alpha)|=1$. Simple calculations show that $Fix(V_\alpha)=\{C\}$. 
\end{proof}

We can easily check a local behavior of the fixed point $C=\left(\frac{1}{3},\frac{1}{3},\frac{1}{3}\right)$.

\begin{proposition}\label{localbehavioroffixedpoint}
Let $V_\alpha :S^2\to S^2$ be a quadratic stochastic operator given by \eqref{V_alpha}, where $\alpha\in(0,1)$. Then the following statements hold true:
\begin{itemize}
\item[(i)] If $\alpha\neq\frac{1}{2}$ then the fixed point $C=\left(\frac{1}{3},\frac{1}{3},\frac{1}{3}\right)$ is repelling;
\item[(ii)] If $\alpha=\frac{1}{2}$ then the fixed point $C=\left(\frac{1}{3},\frac{1}{3},\frac{1}{3}\right)$ is non-hyperbolic.
\end{itemize}
\end{proposition}

We shall separately study two cases $\alpha\neq\frac{1}{2}$ and $\alpha=\frac{1}{2}$.

\begin{theorem}
Let $V_\alpha :S^2\to S^2$ be a quadratic stochastic operator given by \eqref{V_alpha}, where $\alpha\neq\frac{1}{2}$. Then $\omega(x^0)\subset intS^2$ is an infinite compact subset of $intS^2$ for any $x^0\neq C$.
\end{theorem}

\begin{proof}
Let $\alpha\neq\frac{1}{2}$. Since $V_\alpha$ is continuous and $V_\alpha(S^2)\subset intS^2$, an omega limiting set $\omega(x^0)$ is a nonempty compact set and $\omega(x^0)\subset intS^2$, for any  $x^0\neq C$. We want to show that $\omega(x^0)$ is infinite for any $x^0\neq C.$ Since $C$ is repelling, we have that $C\notin \omega(x^0)$. Let us pick up any point $x^{*}\in\omega(x^0)$ from the set $\omega(x^0)$. Since the operator $V_\alpha$ does not have any periodic point, the trajectory $\left\{V^{(n)}_\alpha(x^{*})\right\}_{n=1}^\infty$ of the point $x^{*}$ is infinite. Since $V_\alpha$ is continuous, we have that $\left\{V^{(n)}_\alpha(x^{*})\right\}_{n=1}^\infty\subset \omega(x^0).$ This shows that $\omega(x^0)$ is infinite for any $x^0\neq C.$    
\end{proof}

\begin{remark}
It is worth mentioning  that the sets of omega limiting  points $\omega_{V_0}(x^0)$ and $\omega_{V_1}(x^0)$ of both operators $V_0$ and $V_1$ are infinite. However, unlike the operator $V_\alpha$, we have inclusions $\omega_{V_0}(x^0)\subset \partial S^2$ and $\omega_{V_1}(x^0)\subset \partial S^2$. Moreover, both operators $V_0$ and $V_1$ are non-ergodic.
\end{remark}

\begin{conjecture}
Let $V_\alpha :S^2\to S^2$ be a quadratic stochastic operator given by \eqref{V_alpha}, where $\alpha\neq\frac{1}{2}$. Then the following statements hold true:

\begin{itemize}
\item[(i)] {The operator $V_\alpha$ is non-ergodic;}

\item[(ii)] {The operator $V_\alpha$ exhibits a Li-Yorke chaos.}
\end{itemize}

\end{conjecture}

Now, we shall study the case $\alpha=\frac{1}{2}.$ The operator $V_{1/2}:S^2\to S^2$ takes the following form
\begin{eqnarray}\label{V_12}
V_{1/2}:
\begin{cases}
x'_1=\frac{1}{2}x_1^2+2x_1x_2+\frac{1}{2} x_2^2\\
x'_2=\frac{1}{2}x_2^2+2x_2x_3+\frac{1}{2} x_3^2\\
x'_3=\frac{1}{2}x_3^2+2x_3x_1+\frac{1}{2} x_1^2
\end{cases}
\end{eqnarray}

In this case, the fixed point $C=\left(\frac{1}{3},\frac{1}{3},\frac{1}{3}\right)$ is non-hyperbolic and the spectrum of the Jacobian of the operator $V_{1/2}$ at the fixed point $C$ is $Sp(J(C))=\left\{2,\frac{1\pm \sqrt{3}i}{2}\right\}$.  

Let us define the following sets
\begin{eqnarray*}
l_1=\{x\in S^2: x_2=x_3\}, \ l_2=\{x\in S^2: x_1=x_3\}, \ l_3=\{x\in S^2: x_1=x_2\},\\
S_1=\{x\in S^2: x_1\geq x_2\geq x_3\}, \quad S_2=\{x\in S^2: x_1\geq x_3\geq x_2\},\\ 
S_3=\{x\in S^2: x_3\geq x_1\geq x_2\}, \quad S_4=\{x\in S^2: x_3\geq x_2\geq x_1\},\\
S_5=\{x\in S^2: x_2\geq x_3\geq x_1\}, \quad S_6=\{x\in S^2: x_2\geq x_1\geq x_3\}.
\end{eqnarray*}

\begin{proposition}
We have the following cycles:
\begin{itemize}
\item[(i)] $l_1\overset{V_{1/2}}{\longrightarrow} l_2\overset{V_{1/2}}{\longrightarrow} l_3\overset{V_{1/2}}{\longrightarrow} l_1;$
\item[(ii)] $S_1\overset{V_{1/2}}{\longrightarrow} S_2\overset{V_{1/2}}{\longrightarrow} S_3\overset{V_{1/2}}{\longrightarrow} S_4\overset{V_{1/2}}{\longrightarrow} S_5 \overset{V_{1/2}}{\longrightarrow} S_6\overset{V_{1/2}}{\longrightarrow} S_1;$
\end{itemize}
\end{proposition}

\begin{proof}
Let $V_{1/2}$ be an operator given by \eqref{V_12}. One can easily check that
\begin{eqnarray}\label{mainequality}
\begin{array}{c}
x'_1-x'_2=(x_1-x_3)\cfrac{1+3x_2}{2}\\
x'_1-x'_3=(x_2-x_3)\cfrac{1+3x_1}{2}\\
x'_2-x'_3=(x_2-x_1)\cfrac{1+3x_3}{2}
\end{array}
\end{eqnarray}
The proof the proposition follows from the above equality.
\end{proof}

\begin{theorem}
Let $V_{1/2} :S^2\to S^2$ be a quadratic stochastic operator given by \eqref{V_12}. The following statements hold true:
\begin{itemize}
\item[(i)] $\phi(x)=|x_1-x_2||x_1-x_3||x_2-x_3|$ is a Lyapunov function;
\item[(ii)] The trajectory always converges to the fixed point $C=\left(\frac{1}{3},\frac{1}{3},\frac{1}{3}\right)$. 
\end{itemize}
\end{theorem}

\begin{proof}
(i). Let $V_{1/2}$ be an operator given by \eqref{V_12}. It follows from \eqref{mainequality} that
$$
\phi(V_{1/2}(x))=\phi(x)\frac{1+3x_1}{2}\frac{1+3x_2}{2}\frac{1+3x_3}{2}.
$$
On the other hand, we have that
$$
\frac{1+3x_1}{2}\frac{1+3x_2}{2}\frac{1+3x_3}{2}\leq \left(\frac{\frac{1+3x_1}{2}+\frac{1+3x_2}{2}+\frac{1+3x_3}{2}}{3}\right)^3=1.
$$
Therefore, one has that $\phi(V(x))\leq \phi(x)$ for any $x\in S^2$. This means that $\phi$ is decreasing a long the trajectory of $V_{1/2}$. Consequently, $\phi$ is a Lyapunov function.

(ii). We know that $\left\{\phi\left(V^{(n)}_{1/2}(x)\right)\right\}_{n=1}^\infty$ is a decreasing bounded sequence. Therefore, the limit $\lim\limits_{n\to\infty}\phi\left(V^{(n)}_{1/2}(x)\right)=\lambda$ exists. We want to show that $\lambda=0$. Suppose that $\lambda\neq 0$. It means that $\overline{\left\{V^{(n)}_{1/2}(x)\right\}_{n=1}^\infty}\subset S^2\setminus\{l_1\cup l_2\cup l_3\}$. Since $\lambda \neq 0$, we get that 
\begin{eqnarray*}
1=\lim\limits_{n\to\infty}\frac{\phi\left(V^{(n+1)}_{1/2}(x)\right)}{\phi\left(V^{(n)}_{1/2}(x)\right)}=\lim\limits_{n\to\infty}\left(\frac{1+3x_1^{(n)}}{2}\frac{1+3x_2^{(n)}}{2}\frac{1+3x_3^{(n)}}{2}\right)
\end{eqnarray*}

On the other hand, since $\overline{\left\{V^{(n)}_{1/2}(x)\right\}_{n=1}^\infty}\subset S^2\setminus\{l_1\cup l_2\cup l_3\}$, there exists $\varepsilon_0$ such that for any $n$
one has that 
$$
\frac{1+3x_1^{(n)}}{2}\frac{1+3x_2^{(n)}}{2}\frac{1+3x_3^{(n)}}{2}<1-\varepsilon_0.
$$

This is a contradiction. It shows that $\lambda=0$. 

Therefore, $\omega(x^0)\subset l_1\cup l_2\cup l_3.$ We want to show that $\omega(x^0)= l_1\cap l_2\cap l_3$.

We know that $\left|x_1^{(n)}-x_2^{(n)}\right|\left|x_1^{(n)}-x_3^{(n)}\right|\left|x_2^{(n)}-x_3^{(n)}\right|\overset{n\to\infty}{\longrightarrow}0$. It follows from \eqref{mainequality} that 
$$
\max\left\{\left|x_1^{(n)}-x_2^{(n)}\right|,\left|x_1^{(n)}-x_3^{(n)}\right|,\left|x_2^{(n)}-x_3^{(n)}\right|\right\}\overset{n\to\infty}{\longrightarrow}0.
$$
This means that $\left(x_1^{(n)},x_2^{(n)},x_3^{(n)}\right)\overset{n\to\infty}{\longrightarrow} \left(\frac{1}{3},\frac{1}{3},\frac{1}{3}\right)$. This completes the proof.
\end{proof}

\subsection{Numerical Results on Dynamics of $V_\alpha$}
We are going to present some pictures of attractors (an omega limiting set) of the operator $V_\alpha:S^2\to S^2$ given by \eqref{V_alpha}.
 
In the cases $\alpha=0$ and $\alpha=1$, the operators $V_0$ and $V_1$ have similar behaviors. The trajectories of both operators, $V_0$ and $V_1$, look as spirals, but one of them is moving by clockwise and another one is moving by anticlockwise. In these cases, we have that  $\omega_{V_0}(x^0)\subset \partial S^2$ and $\omega_{V_1}(x^0)\subset \partial S^2$.

We are interested in the dynamics of the mutation operator $V_\alpha$ while $\alpha$ approaches to $\frac{1}{2}$ from both left and right sides. In order to see some symmetry, we shall provide attractors of $V_\alpha$ and $V_{1-\alpha}$ at the same time.      
 
If we slightly change $\alpha$ from $0$ and $1$, then we can see that the omega limiting set splits from the boundary. Moreover, in the picture, we can see \textbf{"three vertexes"} (it is roughly speaking) where the trajectory shall spend almost all its time around those points. Therefore, we are expecting non-ergodicity of the operator $V_\alpha$ if $\alpha$ is near by $0$ or $1$. 

If $\alpha$ becomes close to the $\frac{1}{2}$ then we can see some chaotic pictures. We observe from the pictures that, in the cases $\alpha$ and $1-\alpha$, the attractors are the same but different from the orientations. One of them is moving by clockwise and another one is by anticlockwise. If $\alpha$ is enough close to $\frac{1}{2}$ then we detect completely different pictures. 
Here are some pictures for the values of $\alpha= 0.4995, 0.4999, 0.5005, 0.5001$. In this cases, attractors are disconnected sets and they consist of 6 connected sets. It is very surprising that why the number of connected sets are 6 and why not 5 or 7.

\begin{figure}[h!]\label{Figalpha0109}
\begin{center}
\includegraphics[height=4cm,width=4cm]{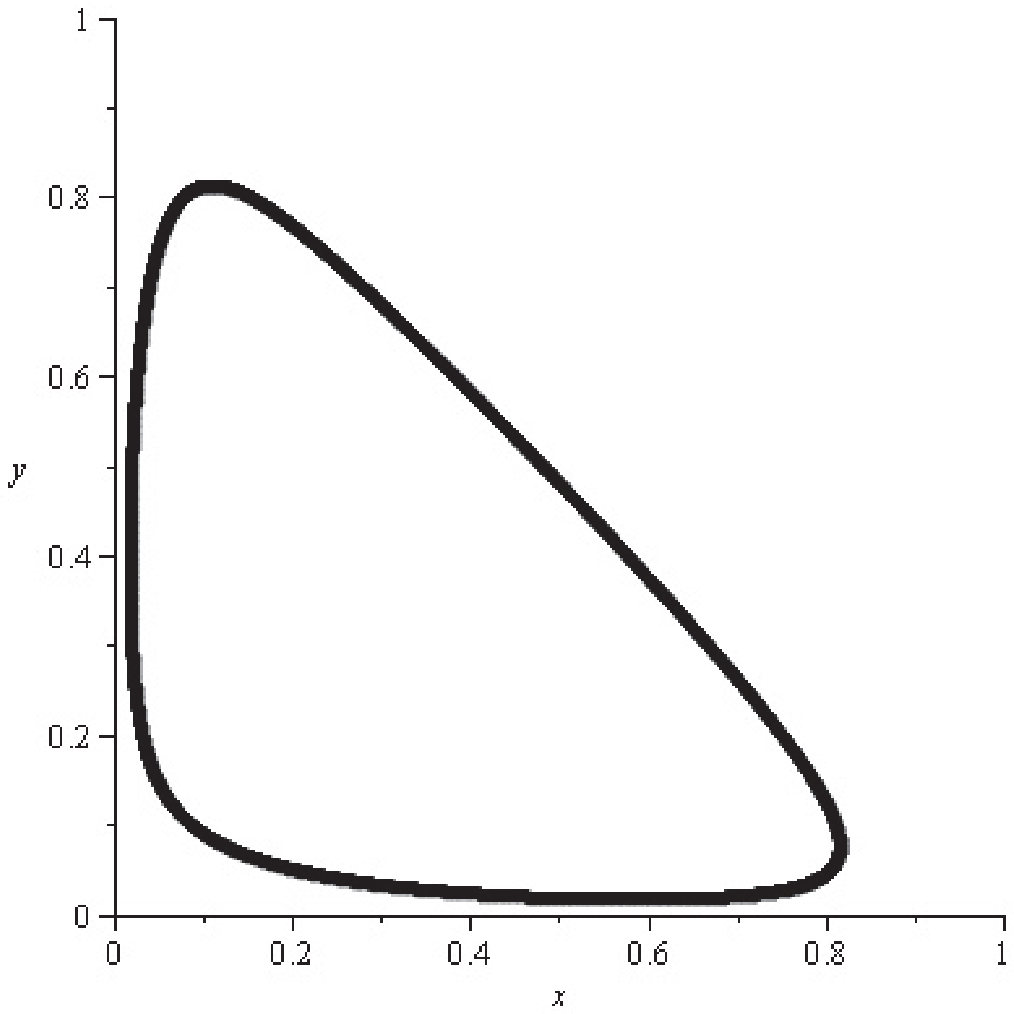} \quad \quad \quad \quad \quad 
\includegraphics[height=4cm,width=4cm]{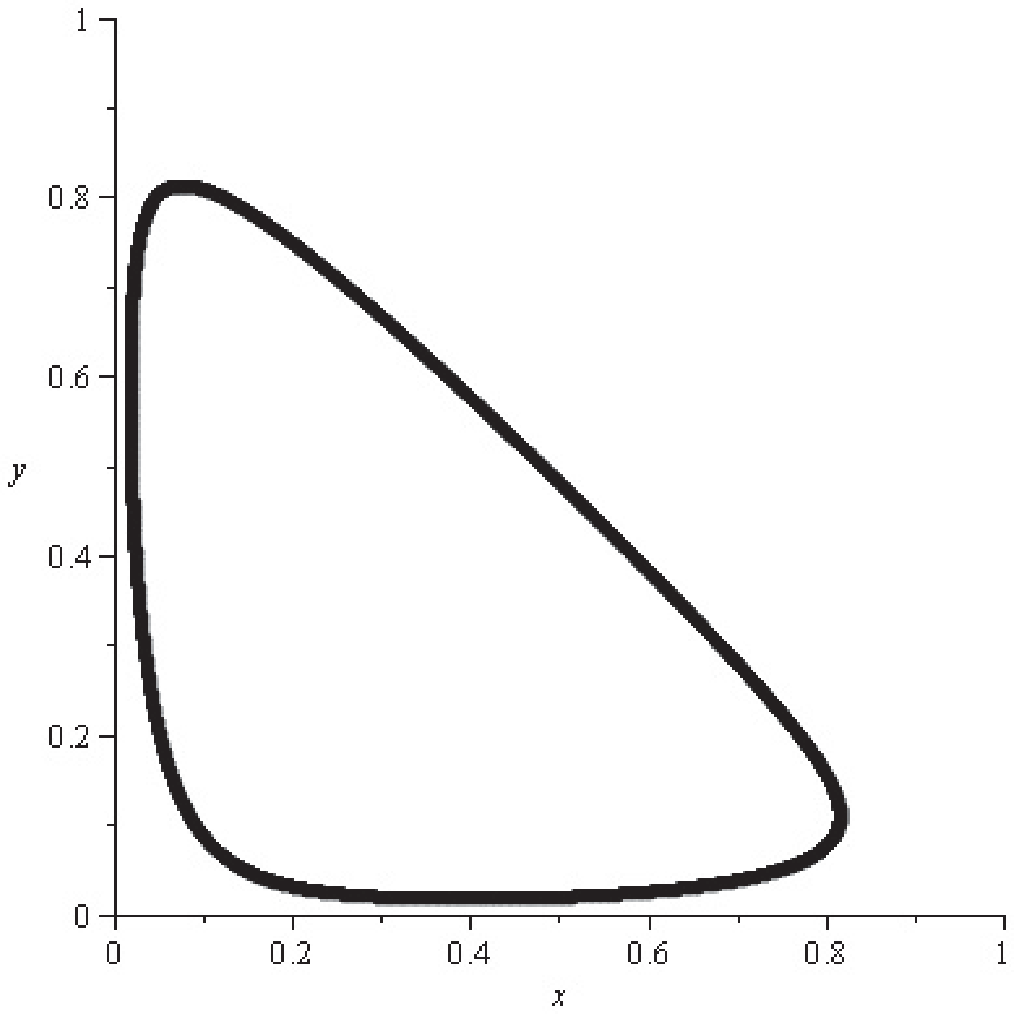}
\caption{Attractors of $V_\alpha$: $\alpha=0.1$ and $\alpha=0.9$}
\end{center}
\end{figure}

\begin{figure}[h!]\label{Figalpha04970503}
\begin{center}
\includegraphics[height=4cm,width=4cm]{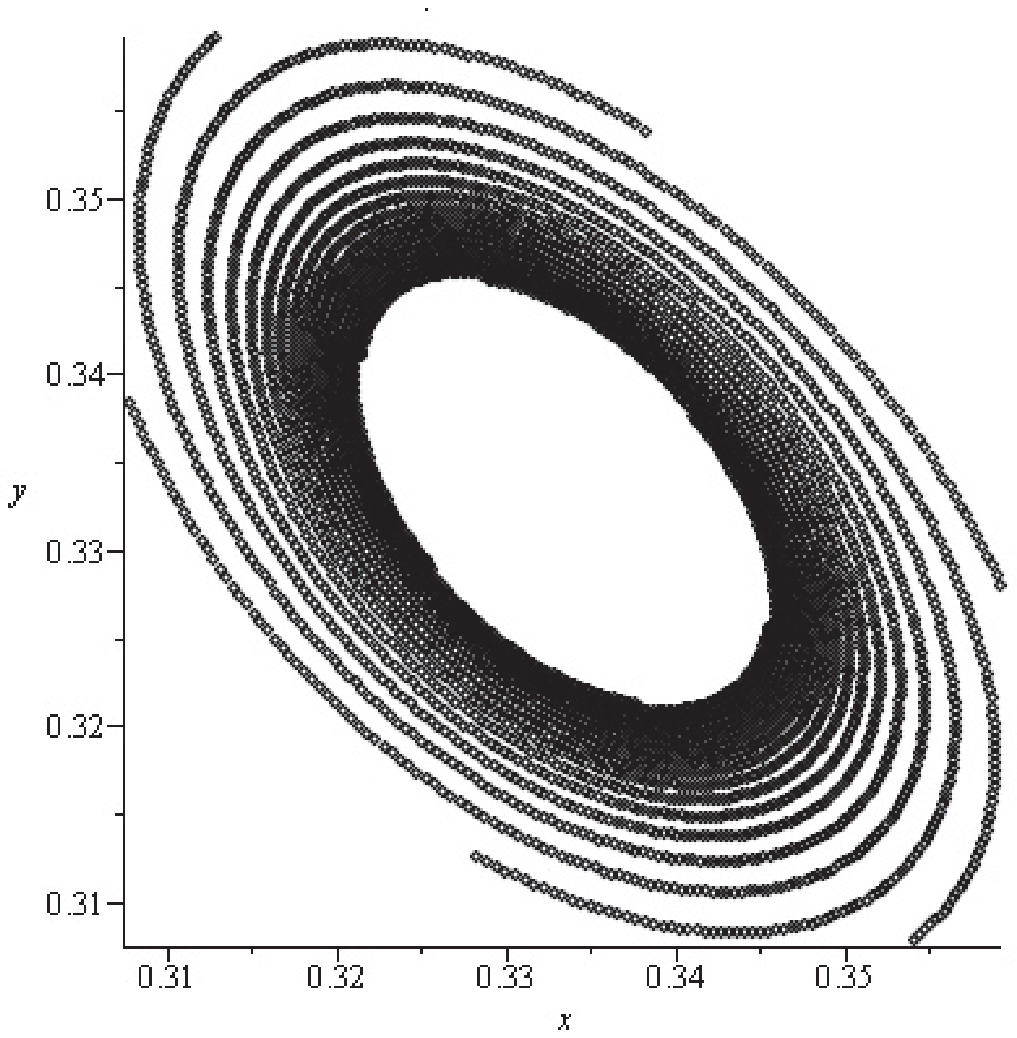} \quad \quad \quad \quad \quad 
\includegraphics[height=4cm,width=4cm]{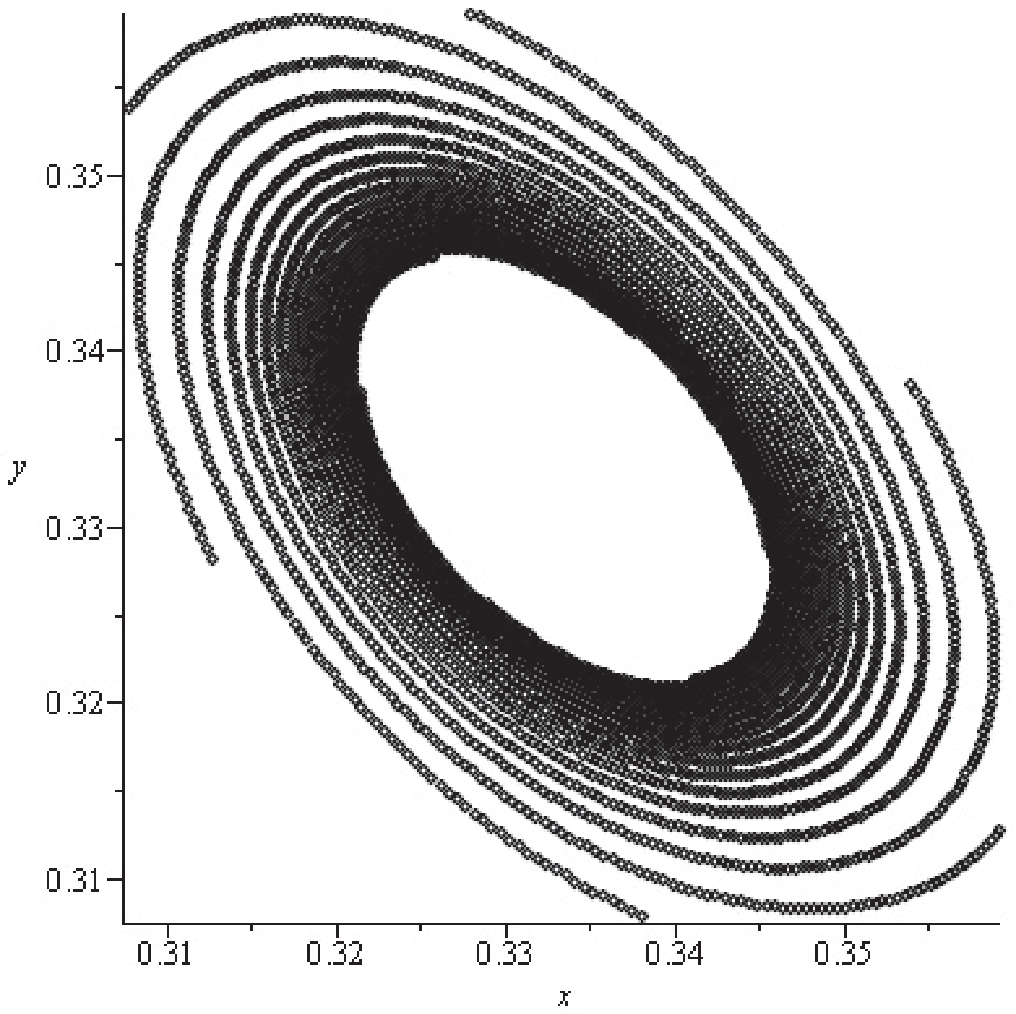}
\caption{Attractors of $V_\alpha$: $\alpha=0.497$ and $\alpha=0.503$}
\end{center}
\end{figure}

\begin{figure}[h!]\label{Figalpha04990501}
\begin{center}
\includegraphics[height=4cm,width=4cm]{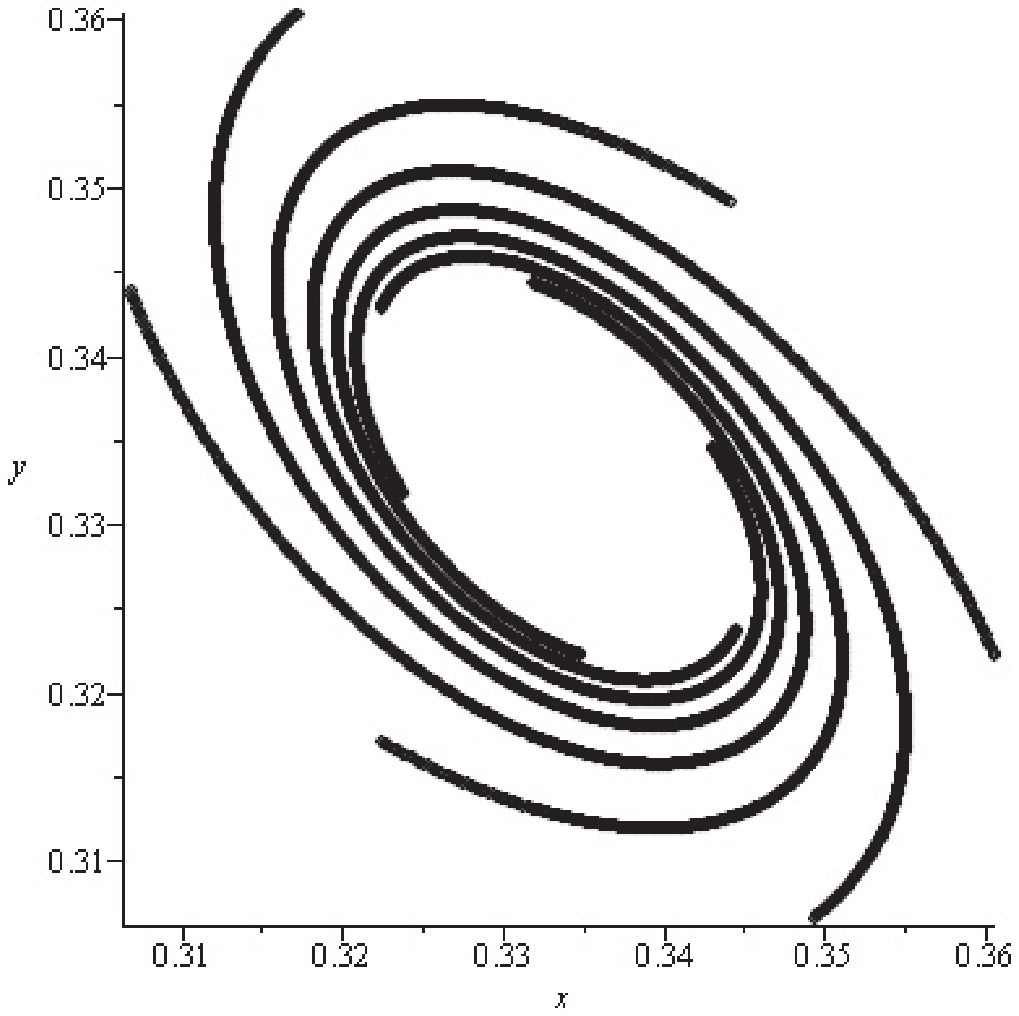} \quad \quad \quad \quad \quad 
\includegraphics[height=4cm,width=4cm]{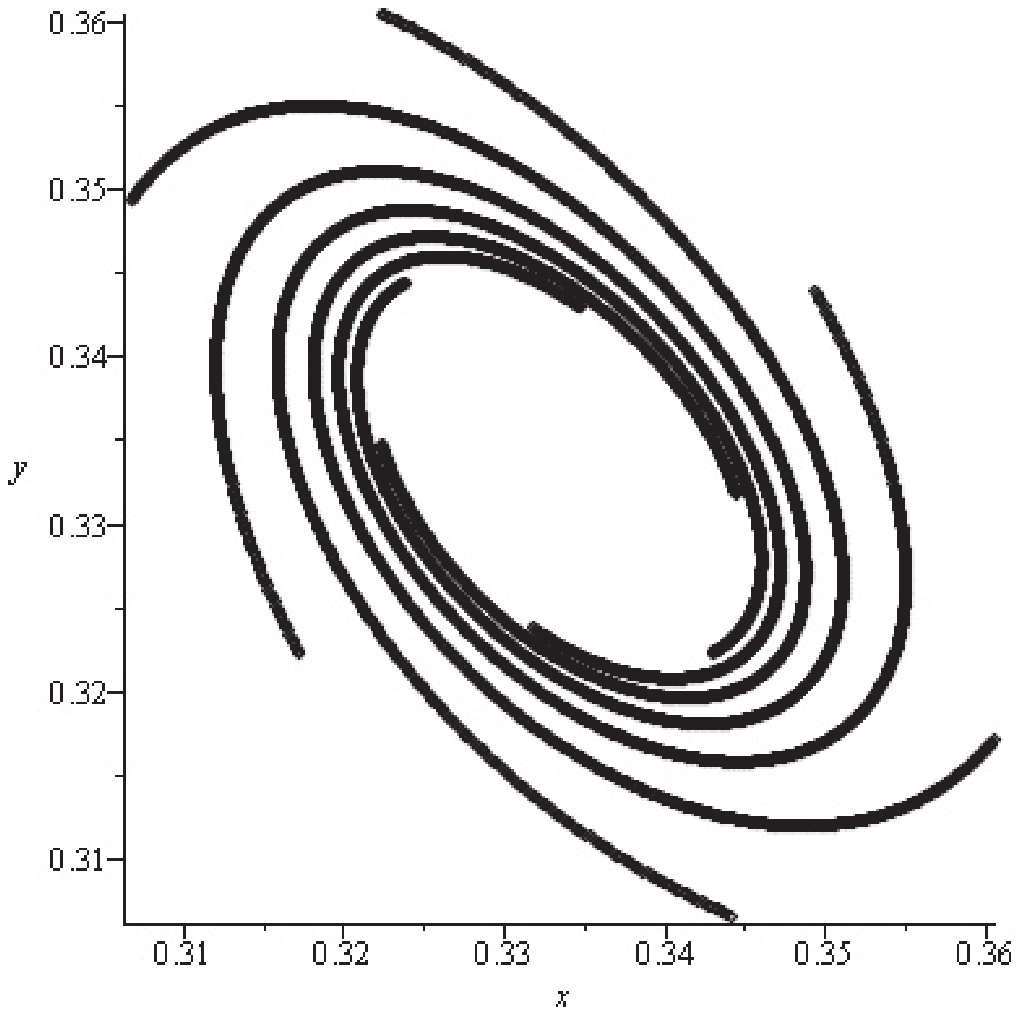}
\caption{Attractors of $V_\alpha$: $\alpha=0.499$ and $\alpha=0.501$}
\end{center}
\end{figure}

\begin{figure}[h!]\label{Figalpha0499505005}
\begin{center}
\includegraphics[height=4cm,width=4cm]{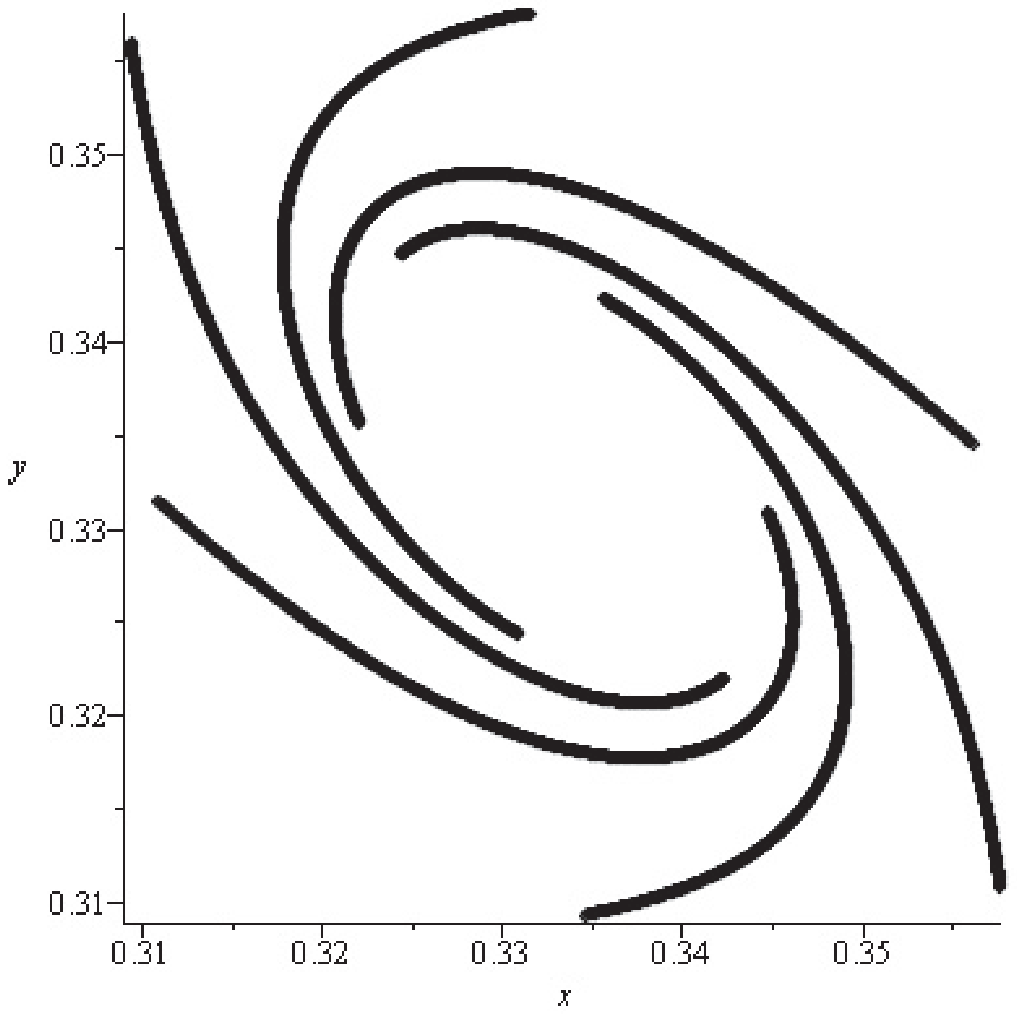} \quad \quad \quad \quad \quad 
\includegraphics[height=4cm,width=4cm]{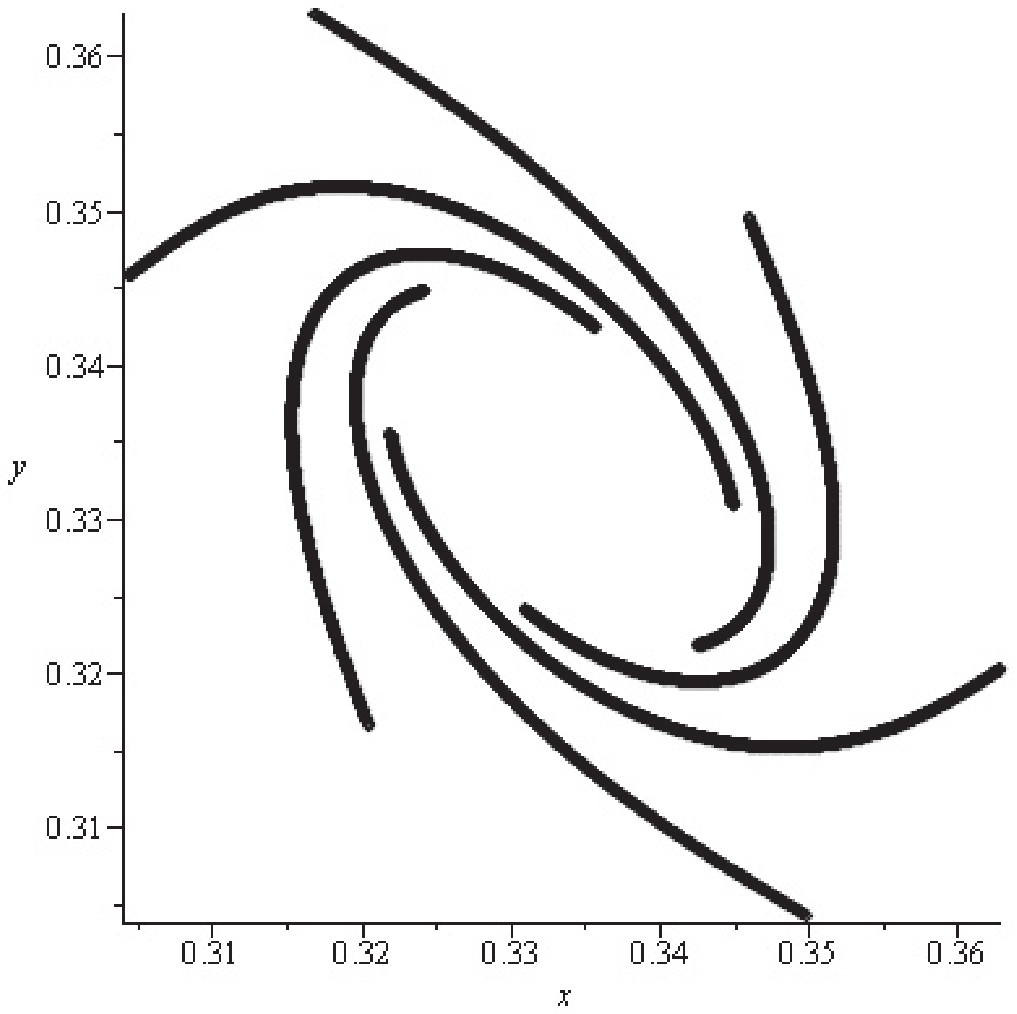}
\caption{Attractors of $V_\alpha$: $\alpha=0.4995$ and $\alpha=0.5005$}
\end{center}
\end{figure}

\begin{figure}[h!]\label{Figalpha0499905001}
\begin{center}
\includegraphics[height=4cm,width=4cm]{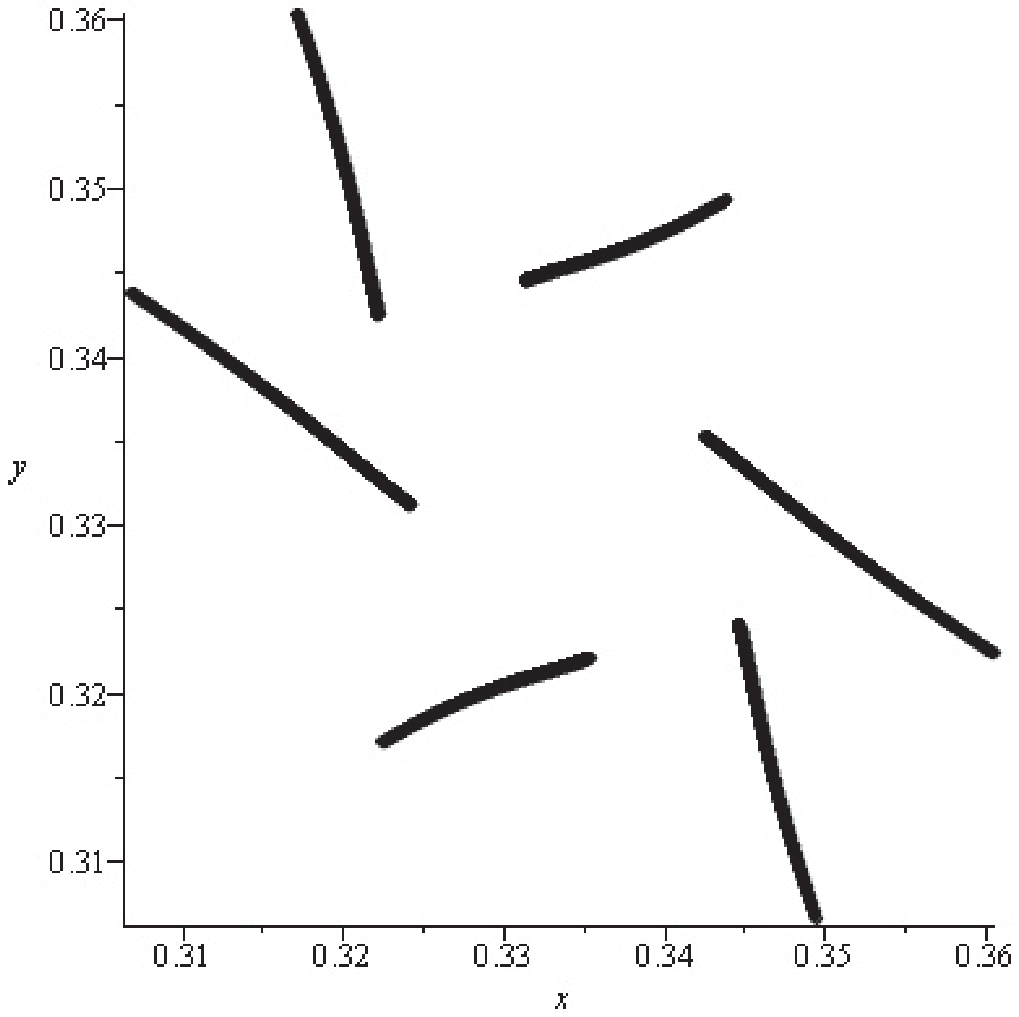} \quad \quad \quad \quad \quad 
\includegraphics[height=4cm,width=4cm]{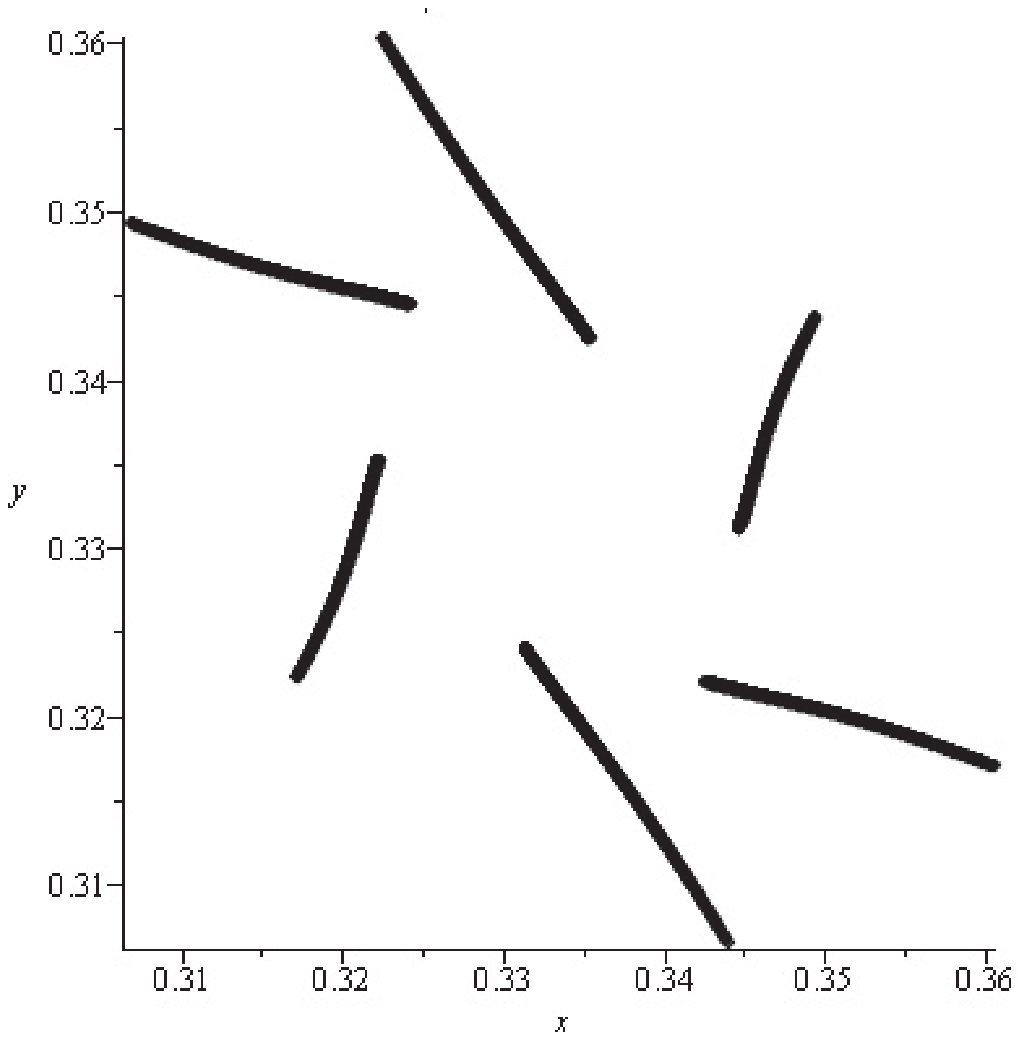}
\caption{Attractors of $V_\alpha$: $\alpha=0.4999$ and $\alpha=0.5001$}
\end{center}
\end{figure}  
   
For the operator \eqref{V_alpha}, the fluctuation point is $\alpha=\frac{1}{2}$. In this case, the influence of the chaotic operators $V_0$ and $V_1$ are the same. Therefore, the operator $V_{0.5}$ becomes regular and ergodic. This completes the numerical study of the operator \eqref{V_alpha}. 

\subsection{Analytic Results on Dynamics of $W_\alpha$} We are aiming to present some analytic results on dynamics of $W_\alpha:S^2\to S^2$:
\begin{eqnarray}\label{W_alpha}
W_\alpha:
\begin{cases}
x'_1=(1-\alpha)x_1^2+2x_1x_2+\alpha x_3^2\\
x'_2=(1-\alpha)x_2^2+2x_2x_3+\alpha x_1^2\\
x'_3=(1-\alpha)x_3^2+2x_3x_1+\alpha x_2^2
\end{cases},
\end{eqnarray}
where $W_\alpha(x)=x'=(x'_1,x'_2,x'_3)$ and $0<\alpha<1$. As we already mentioned, this operator can be written in the following form: $W_\alpha=(1-\alpha)W_0+\alpha W_1$ for any $0<\alpha<1$, where
$$
W_0:
\begin{cases}
x'_1=x_1^2+2x_1x_2\\
x'_2=x_2^2+2x_2x_3\\
x'_3=x_3^2+2x_3x_1
\end{cases}, \quad \quad \quad 
W_1:
\begin{cases}
x'_1=x_2^2+2x_1x_2\\
x'_2=x_3^2+2x_2x_3\\
x'_3=x_1^2+2x_3x_1
\end{cases}.
$$

As we already stated, the operator $W_0=V_0$ is Zakharevich's operator \eqref{ExampleofZakharevich} and the operator $W_1$ is a permutation of the operator which was studied in \cite{GR1989}. By means of methods which were used in \cite{GR1989}, we can easily prove the following result.

\begin{proposition} Let $W_1:S^2\to S^2$ be a quadratic stochastic operator given by \eqref{W_alpha} with $\alpha=1$. Then the following statements hold true:
\begin{itemize}
\item[(i)] {The operator $W_1$ has a unique fixed point $C=(\frac{1}{3},\frac{1}{3},\frac{1}{3})$ which is attracting;}

\item[(ii)] The vertexes of the simplex $e_1,e_2,e_3$ are 3-periodic points;

\item[(iii)] {$\phi(x)=x_1^2+x_2^2+x_3^2-\frac{1}{3}$ is a Lyupanov function;}

\item[(iv)] {The operator $W_1$ is regular in $intS^2$.}
\end{itemize}
\end{proposition}

By means of the same methods and techniques which are used for the operator $V_\alpha$, we can prove the following results

\begin{proposition}
Let $W_\alpha :S^2\to S^2$ be a quadratic stochastic operator given by \eqref{W_alpha}. Then it has a unique fixed point $C=(\frac{1}{3},\frac{1}{3},\frac{1}{3})$, i.e., $Fix(W_\alpha)=\{C\}$. Moreover, one has that:

\begin{itemize}
\item[(i)] If $0<\alpha<1-\frac{\sqrt{3}}{2}$ then the fixed point is repelling;

\item[(ii)] If  $1-\frac{\sqrt{3}}{2}<\alpha<1$ then the fixed point is attracting;

\item[(iii)] If $\alpha=1-\frac{\sqrt{3}}{2}$ then the fixed point is non-hyperbolic.
\end{itemize}

\end{proposition}  

\begin{theorem}
Let $W_\alpha :S^2\to S^2$ be a quadratic stochastic operator given by \eqref{W_alpha}. Then the following statements hold true:

\begin{itemize}
\item[(i)] If $0<\alpha<1-\frac{\sqrt{3}}{2}$ then $\omega(x^0)\subset intS^2$ is an infinite compact set for any $x^0\neq C$;

\item[(ii)] If  $1-\frac{\sqrt{3}}{2}\leq\alpha<1$ then $\omega(x^0)=\{C\}$ for any $x^0\in S^2$.

\end{itemize}
\end{theorem}

\begin{conjecture}
Let $W_\alpha :S^2\to S^2$ be a quadratic stochastic operator given by \eqref{W_alpha}, where  $0<\alpha<1-\frac{\sqrt{3}}{2}$. Then the following statements hold true:

\begin{itemize}
\item[(i)] {The operator $W_\alpha$ is non-ergodic;}

\item[(ii)] {The operator $W_\alpha$ exhibits a Li-Yorke chaos.}
\end{itemize}

\end{conjecture}

\subsection{Numerical Results on Dynamics of $W_\alpha$} 
We are going to present some pictures of attractors (an omega limiting set) of the operator $W_\alpha:S^2\to S^2$ given by \eqref{W_alpha}. 

In this cases $\alpha=0$ and $\alpha=1$, the operator $W_0$ is a chaotic operator and the operator $W_1$ is regular. Since $W_\alpha=(1-\alpha)W_0+\alpha W_1$, the mutation operator $W_\alpha$ gives a transition from the chaotic behavior to the regular behavior. Consequently, we are aiming to find \textbf{fluctuation points} in which \textit{we can see the transition from the chaotic behavior to the regular behavior.}

\begin{figure}[h!]\label{WFigalpha00101}
\begin{center}
\includegraphics[height=4cm,width=4cm]{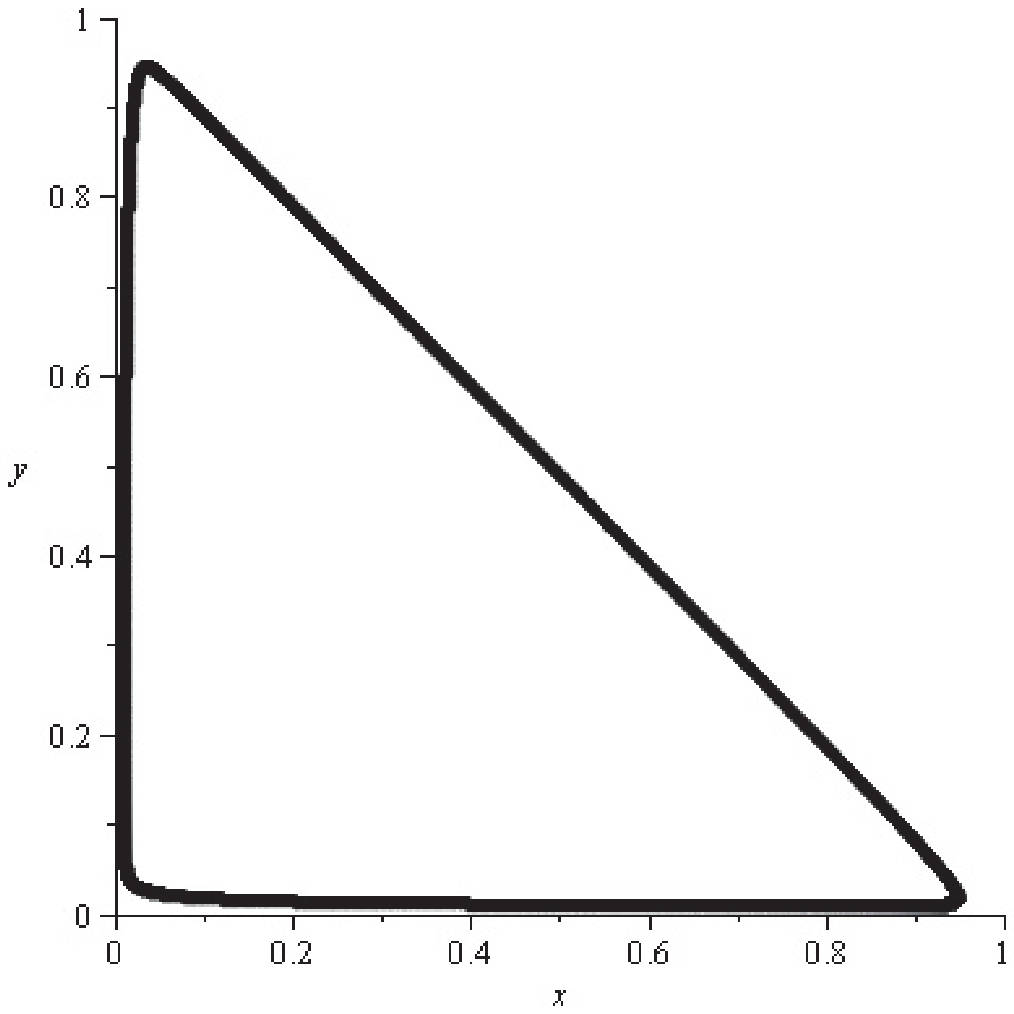} \quad \quad \quad \quad \quad 
\includegraphics[height=4cm,width=4cm]{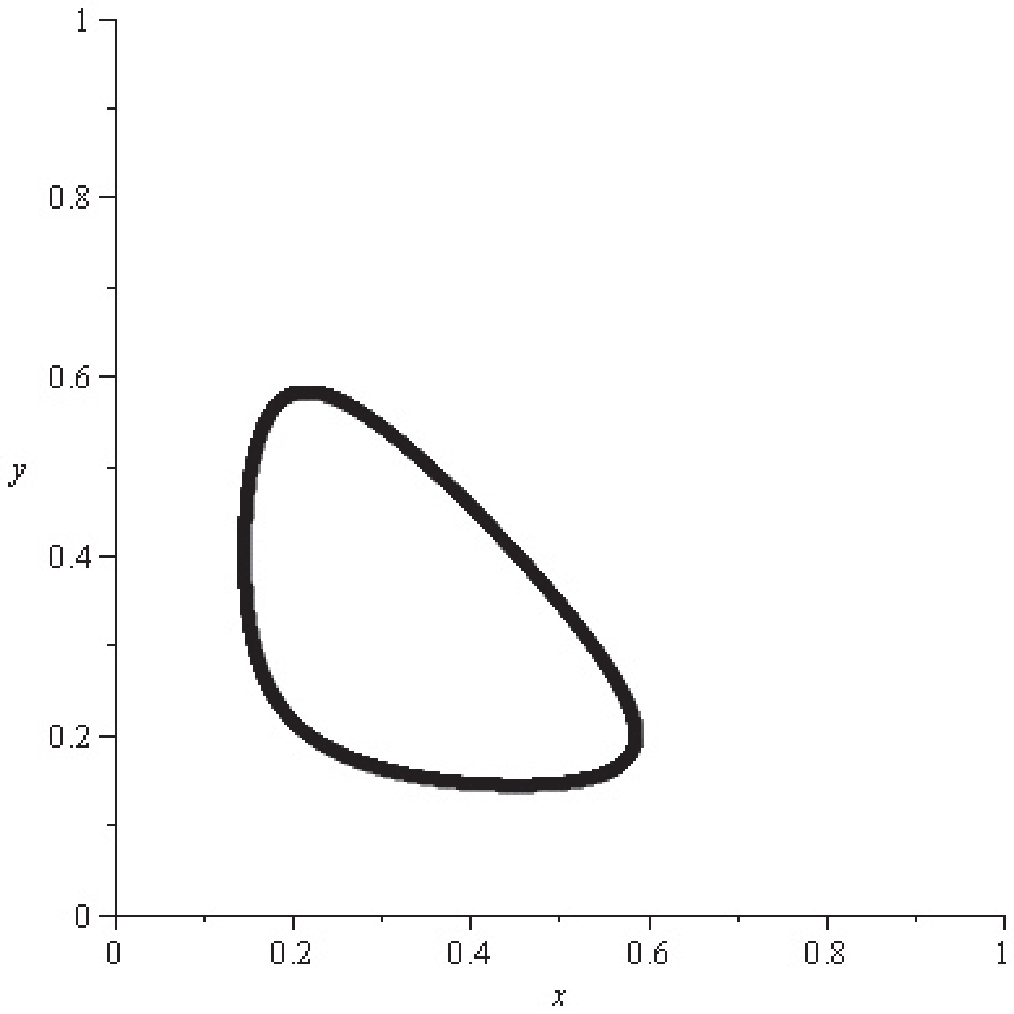}
\caption{Attractors of $W_\alpha$: $\alpha=0.001$ and $\alpha=0.01$}
\end{center}
\end{figure}

If $\alpha$ is very close to $0$ then attractors of the operator $W_\alpha$ are splitted from the boundary of the simplex. However, the influence of the operator $W_0$ is still very high. Therefore, we are expecting that the operator $W_\alpha$ is non-ergodic and chaotic, whenever $\alpha$ is very close to $0$. 

If $\alpha$ becomes to close to $0.13333$ then we can see a different picture. In attractors, we are able to detect some \textbf{"growing hairs"} (it is a just terminology).  Moreover, if we shall continue to increasing $\alpha$ then these "hairs" will start to rise and they eventually become straight. Therefore, \textbf{$\alpha=0.13333$ is the first fluctuation point.} It is very interesting that the number of "hairs" is equal to 12. It is again a question: why do we have 12 numbers of "hears" and why not 11 or 13.

\begin{figure}[h!]\label{WFigalpha01333}
\begin{center}
\includegraphics[height=4cm,width=4cm]{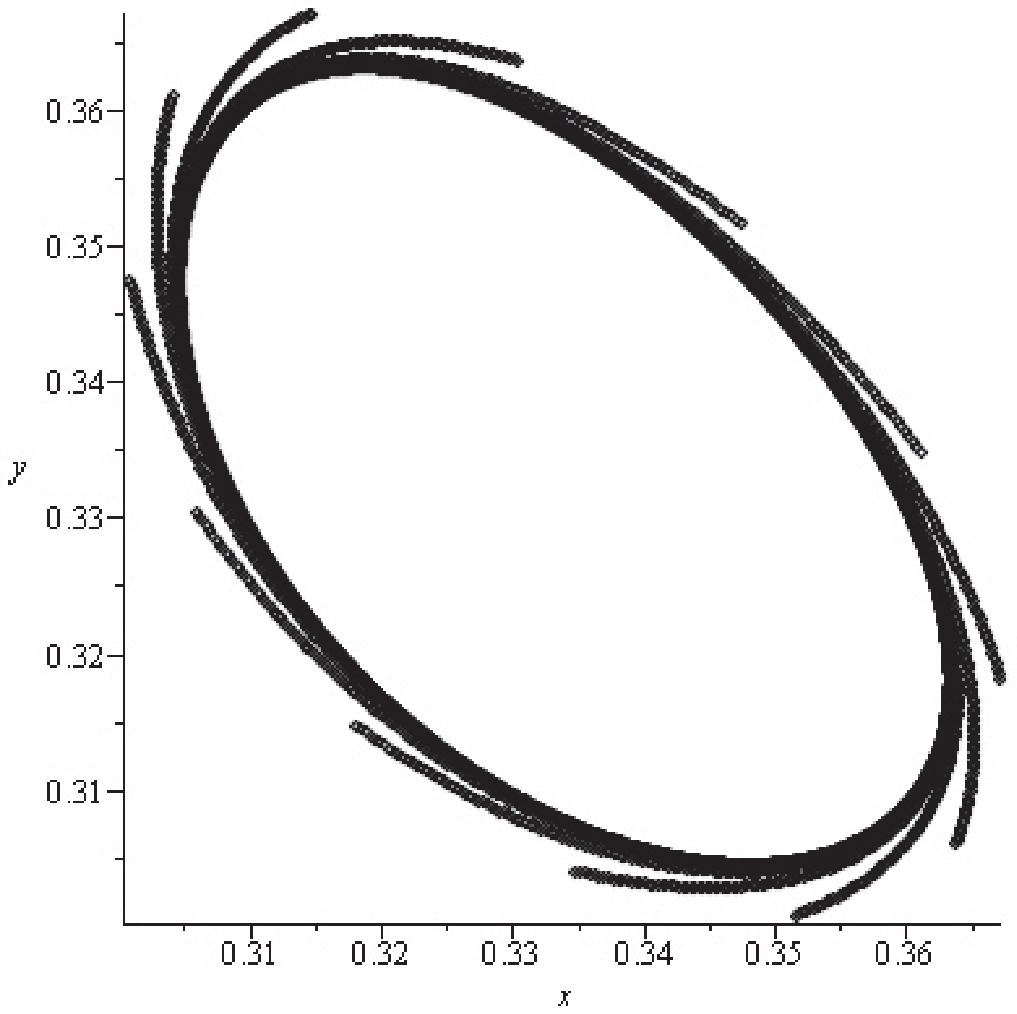} \quad \quad \quad \quad \quad 
\includegraphics[height=4cm,width=4cm]{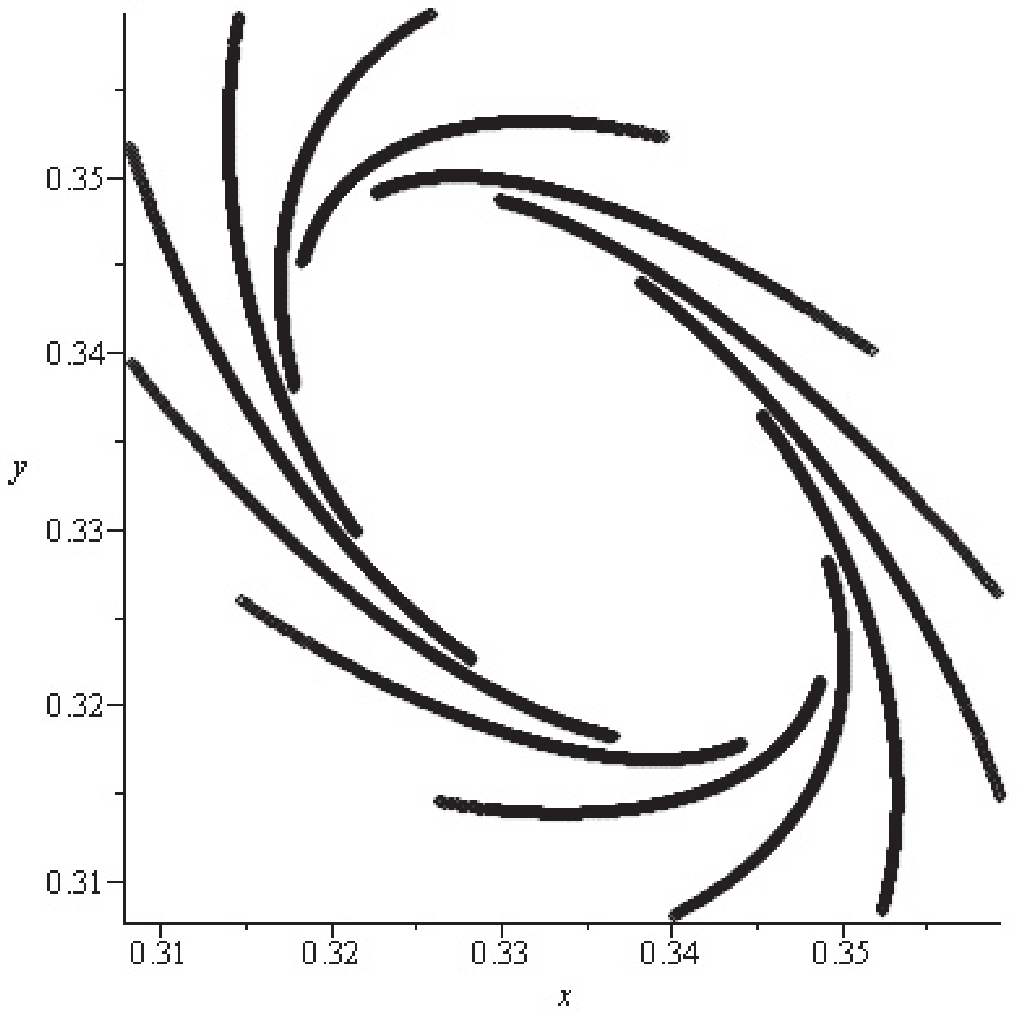}
\caption{Attractors of $W_\alpha$: $\alpha=0.13333$ and $\alpha=0.1338$}
\end{center}
\end{figure}

\begin{figure}[h!]\label{WFigalpha01390144015}
\begin{center}
\includegraphics[height=4cm,width=4cm]{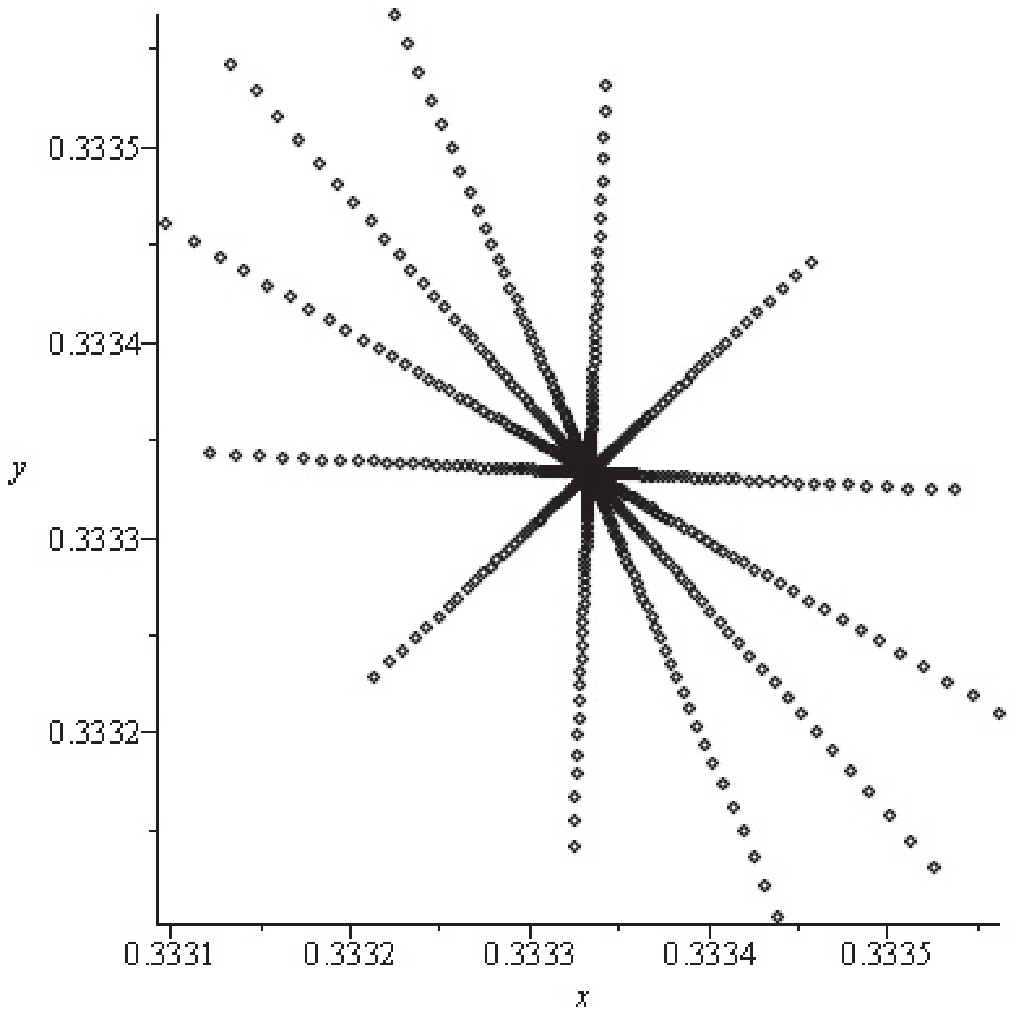} \quad \quad  
\includegraphics[height=4cm,width=4cm]{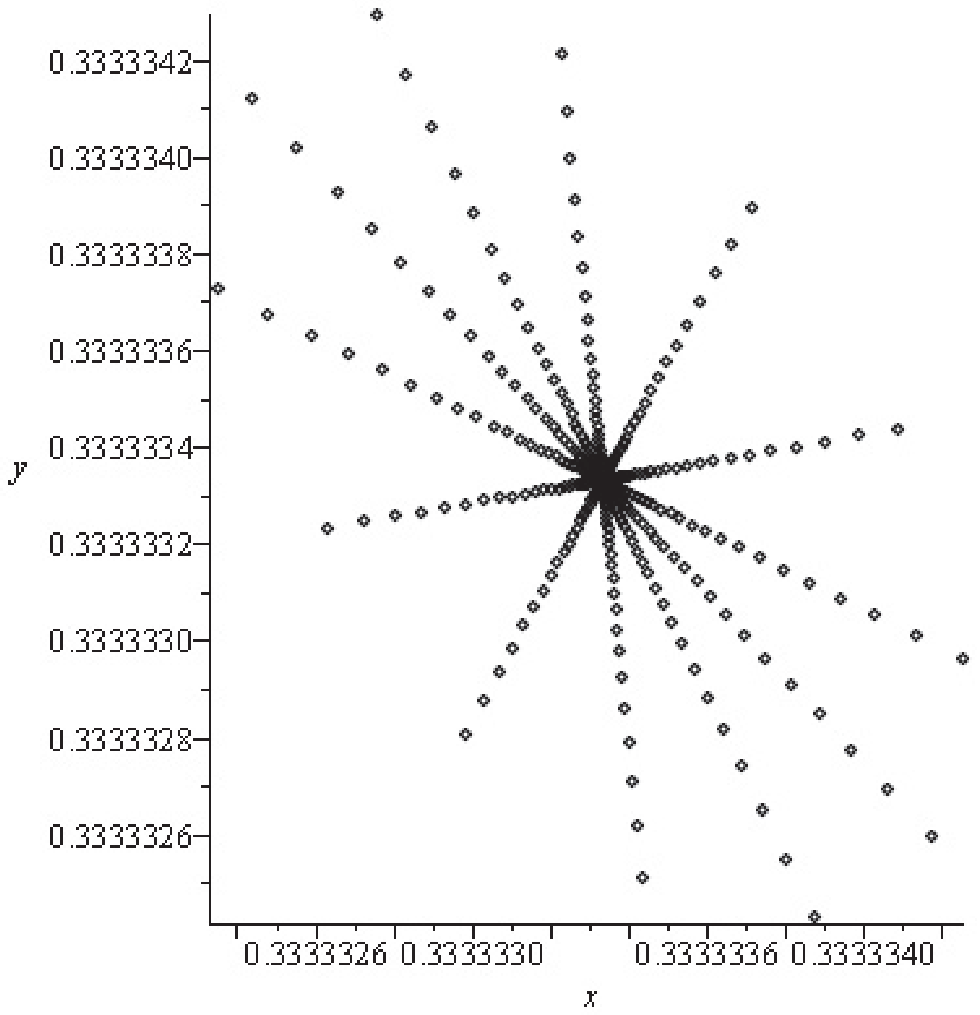}\quad \quad
\includegraphics[height=4cm,width=4cm]{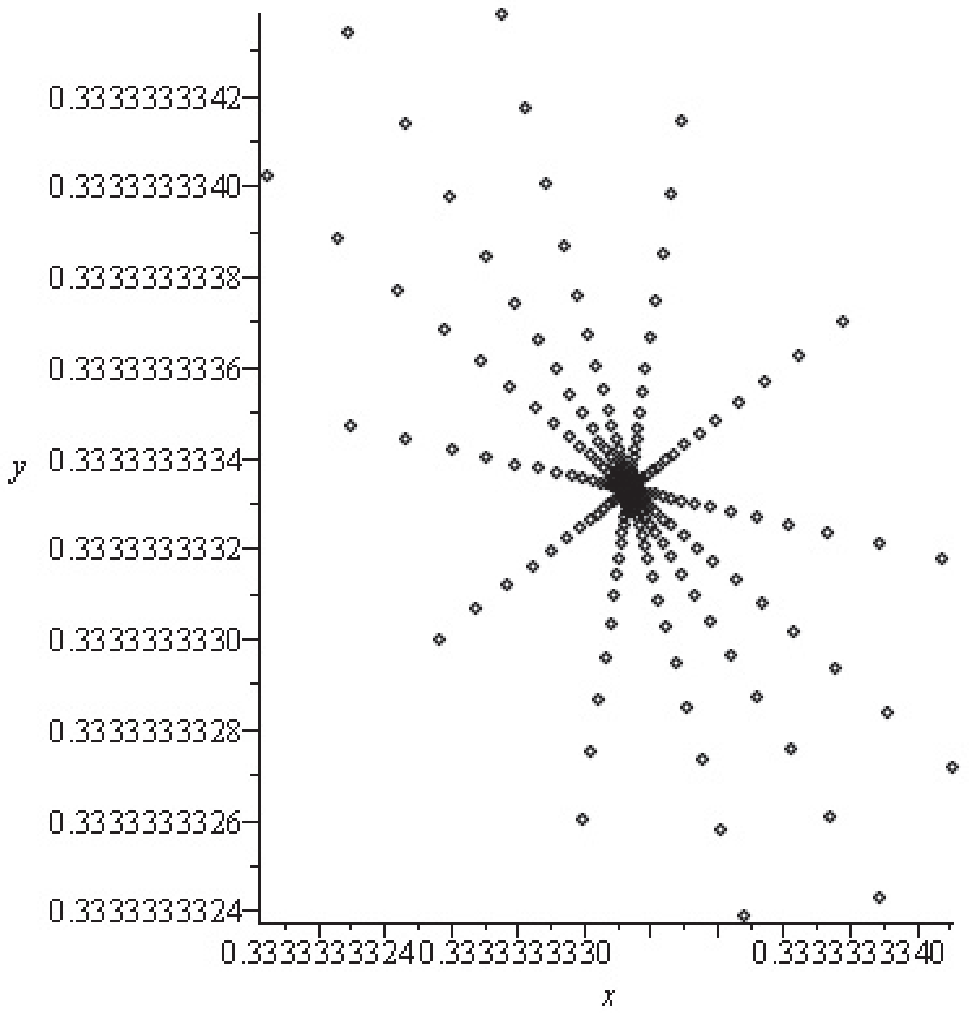}
\caption{Attractors of $W_\alpha$: $\alpha=0.139$, $\alpha=0.144$, and $\alpha=0.150$}
\end{center}
\end{figure} 
  
From these pictures, we can find \textbf{another fluctuation point $\alpha=1-\frac{\sqrt{3}}{2}\simeq 0.1339745$.} Therefore, in order to have a transition from  a chaotic behavior to the regular behavior, we should cross from \textbf{two fluctuation points $\alpha=0.13333$ and $\alpha=1-\frac{\sqrt{3}}{2}\simeq 0.1339745$.}

\section{ Conclusion }

In this paper, we present a mathematical model of the mutation in the biological environment having 3 alleles. We have presented two types of mutations. We have shown that a mutation (a mixing) in the system can be considered as a convex combination of Mendelian inheritances (extreme or non-mixing systems). The first mutation is a convex combination of two Li-Yorke chaotic systems and the second mutation is a convex combination of Li-Yorke chaotic and regular systems. 

In the first case, the first mutation can be considered as an evolution process between two different chaotic  biological systems. In this case, we have shown that there is one fluctuation point (transition point) in order to change one chaotic system into another chaotic system. In the second case, the second mutation can be considered as an evolution process between chaotic and regular biological systems. In this case, we need two fluctuation points (transition points) in order to change the chaotic system into the regular system. We hope that we can find a similar phenomenon in nature.

\end{document}